\documentclass[preprint,11pt]{elsarticle}
\usepackage{lineno,hyperref}
\modulolinenumbers[5]

\usepackage{color}
\usepackage{amsmath}
\usepackage{amsfonts,amsthm,amssymb}
\usepackage{amsfonts}
\usepackage{graphics}
\usepackage{pstricks}
\usepackage{graphicx}
\usepackage{cleveref}
\usepackage{extarrows,chngpage,array,float}
\usepackage{amssymb}
\usepackage{latexsym,bm}
\newtheorem{theorem}{Theorem}[section]
\newtheorem{lemma}[theorem]{Lemma}

\newtheorem{corollary}[theorem]{Corollary}

\allowdisplaybreaks   

\numberwithin{equation}{section}
\journal{}

\begin{document}

\begin{frontmatter}

\title{Coefficients of univariate Tutte polynomials with one variable fixed}

\author{Tianlong Ma$^a$, \ \ Xiaxia Guan$^b$, \ \ Xian'an Jin$^{c,d,}$\footnote{Corresponding author.}, Hong-Jian Lai$^{e,f}$
	\\[1ex]
	\small $^a$School of Science\\[-0.8ex]
	\small Jimei University\\[-0.8ex]
	\small P. R. China
\\[0.3ex]
\small $^b$Department of Mathematics\\[-0.8ex]
\small Taiyuan University of Technology\\[-0.8ex]
\small P. R. China
\\[0.3ex]
\small $^c$School of Mathematics and Statistics\\[-0.8ex]
\small Qinghai Minzu University\\[-0.8ex]
\small P. R. China
\\[0.3ex]
\small $^d$School of Mathematical Sciences\\[-0.8ex]
\small Xiamen University\\[-0.8ex]
\small P. R. China
\\[0.3ex]
\small $^e$School of Mathematics and Systems Science\\[-0.8ex]
\small Guangdong Polytechnic Normal University\\[-0.8ex]
\small P. R. China
\\[0.3ex]
\small $^f$Department of Mathematics \\[-0.8ex]
\small West Virginia University\\[-0.8ex]
\small  USA
\\
\small\tt Email: tianlongma@aliyun.com, guanxiaxia@tyut.edu.cn, xajin@xmu.edu.cn, hjlai2015@hotmail.com}

\begin{abstract}

It is well known that the 2-variable Tutte polynomial of a graph $G$ includes chromatic polynomial and flow polynomial of $G$, i.e. the cases of $y=0$ and $x=0$. In 2013, K\'{a}lm\'{a}n introduced the interior and exterior polynomials which generalized the cases of $y=1$ and $x=1$ of Tutte polynomials of graphs to hypergraphs, and further polymatroids. There have been some results on coefficients of these polynomials, which motivate us to study uniformly the coefficients of $T_M(x,t)$ and $T_M(t,y)$, where $T_M(x,y)$ denotes the Tutte polynomial of a matroid $M$ and $t$ is a fixed real number. In this paper, we introduce two mutually dual parameters $f_k(M)$ and $g_k(M)$ ($g_1(M)$ is the girth of $M$) for any nonnegative integer $k$, and obtain the following results: (1) Formulas for coefficients of the higher-degree terms (related to $g_2(M)$ and $f_2(M)$, respectively) of $T_M(x,t)$ and $T_M(t,y)$ in terms of circuits and hyperplanes of $M$; (2)  when $0\leq t \leq 1$, coefficients of the more higher-degree terms (related to $g_1(M)$ and $f_1(M)$, respectively) of $T_M(x,t)$ and $T_M(t,y)$ are further simplified and characterized; (3) As applications, some known results in the cases $t=0$ and $t=1$ are derived and generalized, and the unimodality of these coefficients in (1) are proved when $t\leq 1$.
\end{abstract}


\end{frontmatter}

\section{Introduction}

Let $X$ be a finite set. Let $M=(X,rk)$ denote a matroid on the ground set $X$ with the rank function $rk$. The rank of $M$, denoted by $rk(M)$, is defined to be $rk(X)$, and for short we always assume that $r=rk(M)$ if $M$ is understood from the context. We refer to \cite{Oxley} for undefined notions on matroids and graphs. The Tutte polynomial of graphs, as a generalization of the chromatic polynomial, was introduced by Tutte \cite{Tutte} in 1954,
and extended to matroids by Crapo \cite{Crapo} in 1969.
The \textit{Tutte polynomial} of a matroid $M=(X,rk)$ is defined by
\begin{align}\label{TutteP}
T_{M}(x,y)=\sum_{A\subseteq X}(x-1)^{r-rk(A)}(y-1)^{|A|-rk(A)}.
\end{align}

For a polynomial $f(x)$ in the variable $x$, let $[x^{i}]f(x)$ denote the coefficient of $x^{i}$ in $f(x)$. In this paper, we study the coefficients of Tutte polynomials of matroids in a unified manner when one of its variables is fixed.
We first recall some concepts of matroids. For a matroid $M=(X,rk)$ and $C\subseteq X$, $C$ is called a \textit{circuit}
of $M$ if for all $e\in C$, $rk(C\setminus \{e\})=|C|-1=rk(C)$.
Let $\mathcal{C}(M)$ denote the set of all circuits of $M$. Set
$$\mathcal{C}_{i}=\{C: C\in \mathcal{C}(M)~\text{and}~|C|=i\}.$$ For $A\subseteq X$, the set $\{e\in X:rk(A\cup \{e\})=rk(A)\}$ is called the \textit{closure} of $A$, denoted by $cl_M(A)$. Further, a subset $F$ of $X$ for which $cl_M(F)=F$ is called a \textit{flat} or a \textit{closed set} of $M$. In particular, a \textit{hyperplane} of $M$ is a flat of rank $r-1$. Let $\mathcal{F}(M)$ denote the set of all flats of a matroid $M$, and let $\mathcal{H}(M)$ denote the set of all hyperplanes of a matroid $M$. Set
$$\mathcal{H}_{i}=\{H: H\in \mathcal{H}(M)~\text{and}~|H|=i\}.$$
Now we define two new parameters $f_k(M)$ and $g_k(M)$.
\vskip0.2cm
\noindent $\bullet$ For any nonnegative integer $k$, define
$$f_k(M)=\max\{|F|: rk(F)=r-k\}$$ unless no such set $F$ exists, in which case we define $f_k(M)=\infty $. Then the following basic properties on $f_k(M)$ are all clear: (1) if $F$ is a maximum subset in $X$ such that $rk(F)=r-k$, then $F\in \mathcal{F}(M)$; (2) $f_0(M)=|X|$ and $f_1(M)$ is equal to the size of the maximum hyperplane in $M$; (3) $f_{k}(M)$ exists if and only if $k\leq r$, and (4) $f_{k+1}(M)<f_{k}(M)$ for $k<r$.

\noindent $\bullet$  For any nonnegative integer $k$, define
\[g_k(M)=\min\{|A|:  rk(A)=|A|-k\}\]
unless no such set $A$ exists, in which case we define $g_k(M)=\infty $. We have the following basic properties on $g_k(M)$: (1) if $A$ is a minimum subset in $X$ such that $rk(A)=|A|-k$, then $A\in \mathcal{G}(M)=\{A\subseteq X: rk(A\setminus \{e\})=rk(A) \text{ for any  }e\in A\}$; (2) $g_0(M)=0$ and $g_1(M)=\min\{|C|: C\in \mathcal{C}(M)\},$ the \textit{girth} $g(M)$ of $M$; (3) $g_{k}(M)$ exists if and only if $k\leq |X|-r$, and (4) $g_{k+1}(M)>g_{k}(M)$ for $k< |X|-r$.

\vskip0.2cm

The \textit{dual} of a matroid $M=(X, rk)$, denoted by $M^{*}$, is defined by the matroid $(X,rk^{*})$, where $rk^{*}(A)=|A|+rk(X\setminus A)-r$ for $A\subseteq X$.
In fact, $A\in \mathcal{G}(M)$ if and only if $\overline{A}=X-A\in \mathcal{F}(M^*)$ and $f_k(M)+g_k(M^*)=|X|$ for $k\leq r$.

\begin{theorem}\label{d2}
	Let $M=(X,rk)$ be a matroid, $t$ be a real number and $\alpha =|X|-k-1$.
\begin{enumerate}
\item If $r\leq |X|-2$, then
 \begin{enumerate}
\item[(1)] for any $k$ with $k\geq r-g_2(M)+3$,
	\begin{align*}	
[x^k]T_M(x,t)=\binom{\alpha}{r-k}-\sum_{j=0}^{r-k}\binom{\alpha-j}{r-k-j}|\mathcal{C}_{j}|
   +(t-1)\sum_{j=0}^{r-k+1}\binom{\alpha-j}{r-k-j+1}|\mathcal{C}_{j}|.
	\end{align*}
    \item[(2)] for any $k$ with $k\geq r-g_2(M)+1$,
	\begin{align*}	
[x^k]T_M(x,1)=\binom{\alpha}{r-k}-\sum_{j=0}^{r-k}\binom{\alpha-j}{r-k-j}|\mathcal{C}_{j}|.
	\end{align*}
\end{enumerate}
\item If $r\geq 2$, then
\begin{enumerate}
\item[(3)] for any $k$ with $k\geq f_2(M)-r+3$,
	\begin{align*}	
[y^k]T_M(t,y)=\binom{\alpha}{r-1}-\sum_{j=r+k}^{|X|}\binom{\alpha-j}{r-1}|\mathcal{H}_{j}|
   +(t-1)\sum_{j=r+k-1}^{|X|}\binom{\alpha-j}{r-2}|\mathcal{H}_{j}|.
	\end{align*}
\item[(4)] for any $k$ with $k\geq f_2(M)-r+1$,
	\begin{align*}	
[y^k]T_M(1,y)=\binom{\alpha}{r-1}-\sum_{j=r+k}^{|X|}\binom{\alpha-j}{r-1}|\mathcal{H}_{j}|.
	\end{align*}
\end{enumerate}
\end{enumerate}

\end{theorem}


We further establish the following necessary and sufficient conditions on the coefficients of the more higher-order terms.

\begin{theorem}\label{d3} 
Let $M=(X,rk)$ be a matroid and $t$ be a real number with $0\leq t< 1$.
\begin{enumerate}
\item If $r< |X|$, then
\begin{enumerate}
\item[(1)] $k\geq r-g_1(M)+2$ if and only if
  \begin{align*}	
 [x^k]T_M(x,t)=\binom{\alpha}{r-k}.
	\end{align*}
\item[(2)] $k\geq r-g_1(M)+1$ if and only if
  \begin{align*}	
 [x^k]T_M(x,1)=\binom{\alpha}{r-k}.
	\end{align*}
\end{enumerate}
\item If $r>0$, then
\begin{enumerate}
\item[(3)] $k\geq f_1(M)-r+2$ if and only if
	\begin{align*}	
[y^k]T_M(t,y)=\binom{\alpha}{r-1}.
	\end{align*}
\item[(4)] $k\geq f_1(M)-r+1$ if and only if
	\begin{align*}	
[y^k]T_M(1,y)=\binom{\alpha}{r-1}.
	\end{align*}
\end{enumerate}
\end{enumerate}
\end{theorem}

These two theorems generalize several known results. The special case when $t=0$ of Theorem \ref{d2} (1)
can be further simplified to obtain Theorem \ref{chra}, which generalizes Theorem \ref{TK} of Teo and Koh in \cite{Teo} from graphs to matroids. The special case when $t=0$ of Theorem \ref{d3} (1)
establishes necessary and sufficient conditions on the coefficients of characteristic polynomial (Theorem \ref{Chiff}), chromatic polynomial (Corollary \ref{graphChrff}) and flow polynomial (Corollary \ref{graphflowC}). Theorem \ref{d2} (4) can be viewed as a generalization Theorem \ref{CG} of Chen and Guo in \cite{Chen} from graphs to matroids, we also obtain the graphical version of Theorem \ref{d2} (2), see Corollary \ref{dcg}.

The paper is organized as follows. In Section 2 we prove Theorem \ref{d2} and in Section 3, we prove \ref{d3}. In Section 4, we discuss the application of Theorems \ref{d2} and \ref{d3} when $t=0$ and $t=1$. We also prove the unimodality of these coefficients in Theorem \ref{d2} (1) and (3) when $t\leq 1$.

\section{Proof of Theorem \ref{d2}}

From the definition of the dual we can deduce the well-known fact
\begin{align}\label{dual}
	T_M(x,y)=T_{M^{*}}(y,x)
\end{align}
for any matroid $M$. The following lemma is obtained by unit translation of variables.

\begin{lemma}\label{genematroid}
	Let $M=(X,rk)$ be a matroid. Then
	\begin{align}\label{Indset}
		[x^k]T_M(x,y)=\sum_{A\subseteq X}(-1)^{r-rk(A)-k}\binom{r-rk(A)}{k} (y-1)^{|A|-rk(A)},
	\end{align}
and
\begin{align}\label{Spanset}
[y^k]T_M(x,y)=\sum_{A\subseteq X}(-1)^{|A|-rk(A)-k}\binom{|A|-rk(A)}{k} (x-1)^{r-rk(A)}.
	\end{align}
\end{lemma}

\begin{proof}
By Equation (\ref{TutteP}), we have
	\begin{align*}
		T_M(x,y+1)&=\sum_{A\subseteq X}(x-1)^{r-rk(A)}y^{|A|-rk(A)}\\
		&=\sum_{A\subseteq X}\sum_{k=0}^{r-rk(A)}(-1)^{r-rk(A)-k}\binom{r-rk(A)}{k} x^{k}y^{|A|-rk(A)}\\
		&=\sum_{k=0}^{r}x^{k}\sum_{A\subseteq X}(-1)^{r-rk(A)-k}\binom{r-rk(A)}{k} y^{|A|-rk(A)}.
	\end{align*}
Therefore Equation (\ref{Indset}) holds. Equation (\ref{Spanset}) can be obtained by Equations (\ref{Indset}) and (\ref{dual}).
\end{proof}

Let $M=(X,rk)$ be a matroid.
A subset $A\subseteq X$ is \textit{independent} if $rk(A)=|A|$. Let $\tau_i(M)$ denote the number of the independent sets with $i$ elements in $M$. A subset $A\subseteq X$ is called to be \textit{spanning} if $rk(A)=r$. Let
$\sigma_i(M)$ denote the number of the spanning sets with $i$ elements in $M$.

In particular, if $y=1$ in Equation (\ref{Indset}), then
	\begin{align}\label{Indset1}
		[x^k]T_M(x,1)=\sum^{r-k}_{i=0}(-1)^{r-i-k}\binom{r-i}{k}\tau_i(M).
	\end{align}
If $x=1$ in Equation (\ref{Spanset}), then
\begin{align}
[y^k]T_M(1,y)=\sum^{|X|}_{i=r+k}(-1)^{i-r-k}\binom{i-r}{k} \sigma_i(M).
	\end{align}

The following combinatorial identity is not new and can be derived from Equation (4) on page 8 in \cite{Riordan}.
\begin{lemma}\label{identical1}
	Let $m$ be a nonnegative integer. Then, for any nonnegative integers $p$ and $k$ with $p\geq k$,
	$$\sum_{i=0}^{p-k}(-1)^{p-i-k}\binom{p-i}{k}\binom{m}{i}=\binom{m-k-1}{p-k}.$$	
\end{lemma}


We also need the following simple lemma.

\begin{lemma}\label{boundrank}
	 Let $M=(X,rk)$ be a matroid and $k\leq |X|-r$.  If $rk(A)\leq g_k(M)-k-1$ for $A\subseteq X$, then $|A|-rk(A)\leq k-1$.
\end{lemma}
\begin{proof}
 If $|A|-rk(A)\geq k$, then there must be a subset $D\subseteq A$ such that $rk(D)=|D|-k$.
By the non-decreasing of the rank function, we have $$|D|=rk(D)+k\leq rk(A)+k\leq g_k(M)-1,$$ which is a contradiction to the definition of $g_k(M)$.
\end{proof}

Now we are in a position to prove Theorem \ref{d2}.
\begin{proof}
\noindent(1). By Equation (\ref{Indset}) in Lemma \ref{genematroid} we only need to consider $A\subseteq X$ with $r-rk(A)\geq k$. For any $A\subseteq X$ with $r-rk(A)\geq k$, since $k\geq r-g_2(M)+3$, we have $rk(A)\leq g_2(M)-3$. By Lemma \ref{boundrank}, we have $|A|-rk(A)\leq 1$. Set
\begin{align*}
\mathcal{A}_{0}&=\{A\subseteq X: |A|=rk(A), rk(A)\leq r-k\},\\
\mathcal{A}_{1}&=\{A\subseteq X: |A|=rk(A)+1, rk(A)\leq r-k\}.
\end{align*}
Therefore, either $A\in \mathcal{A}_{0}$ or $A\in \mathcal{A}_{1}$ for all $A\subseteq X$ satisfying $r-rk(A)\geq k$. So we have
\begin{align}\label{CX}	
[x^k]T_M(x,y+1)=& \sum_{A\subseteq X, r-rk(A)\geq k}(-1)^{r-rk(A)-k}\binom{r-rk(A)}{k} y^{|A|-rk(A)}\nonumber \\
	=&\sum_{A\in \mathcal{A}_{0}}(-1)^{r-|A|-k}\binom{r-|A|}{k}+y\sum_{A\in \mathcal{A}_{1}}(-1)^{r-|A|+1-k}\binom{r-|A|+1}{k}\nonumber\\
    =&\sum_{A\subseteq X, r-rk(A)\geq k}(-1)^{r-|A|-k}\binom{r-|A|}{k}-\sum_{A\in \mathcal{A}_{1}}(-1)^{r-|A|-k}\binom{r-|A|}{k}\nonumber\\
    &+y\sum_{A\in \mathcal{A}_{1}}(-1)^{r-|A|+1-k}\binom{r-|A|+1}{k}.
\end{align}

\noindent \textbf{Claim 1.} In Equation (\ref{CX}), the condition $r-rk(A)\geq k$ can be discarded.
\begin{proof}
In $\sum_{A\subseteq X, r-rk(A)\geq k}(-1)^{r-|A|-k}\binom{r-|A|}{k}$, suppose that $|A|=i$, we only need consider such $A$'s with $0\leq i\leq r-k$. If $i\leq r-k$, then $r-rk(A)\geq r-|A|\geq k$.
\end{proof}
\noindent \textbf{Claim 2.}
\begin{align*}
|\{A: A\in \mathcal{A}_{1}~\text{and}~|A|=i\}|=\sum_{C\in \mathcal{C}(M)}\binom{|X|-|C|}{i-|C|}=\sum_{j=0}^{i}\binom{|X|-j}{i-j}|\mathcal{C}_{j}|.
\end{align*}
\begin{proof}
Let $A\in \mathcal{A}_{1}$.
	Suppose that there exists $A'\subseteq X$ with $|A'|=|A|$ such that $|A'|-rk(A')\geq 2$. Therefore,
	\[|A|-rk(A)= 1<2\leq |A'|-rk(A')=|A|-rk(A'). \]
	Then $rk(A')<rk(A)$. Therefore $r-rk(A')\geq r-rk(A)\geq k$. Since $k\geq r-g_2(M)+3$, we have $rk(A')\leq g_2(M)-3$. By Lemma \ref{boundrank}, we have $|A'|-rk(A')\leq 1$, a contradiction. Therefore, for any $A'\subseteq X$ with $|A'|=|A|$, we have $|A'|-rk(A')\leq 1$. For any $A'\subseteq X$ with $|A'|=|A|$, if $|A'|-rk(A')= 1$, then $rk(A')=rk(A)$ and
	$r-rk(A')=r-rk(A)\geq k$. Therefore, the number of $A'$'s with $|A'|=|A|=i$ such that $|A'|-rk(A')= 1$ is equal to  $\sum_{C\in \mathcal{C}(M)}\binom{|X|-|C|}{i-|C|}$.
		
	Thus the claim holds.
\end{proof}
Applying \textbf{Claim 1}, \textbf{Claim 2}, exchanging summations and Lemma \ref{identical1} step by step, we have
\begin{align}\label{CX1}	
[x^k]T_M(x,y+1)=&\sum_{i=0}^{r-k}(-1)^{r-i-k}\binom{r-i}{k}\binom{|X|}{i}-\sum_{A\in \mathcal{A}_{1}}(-1)^{r-|A|-k}\binom{r-|A|}{k}\nonumber\\
	&+y\sum_{A\in \mathcal{A}_{1}}(-1)^{r-|A|+1-k}\binom{r-|A|+1}{k}\nonumber\\
   =&\sum_{i=0}^{r-k}(-1)^{r-i-k}\binom{r-i}{k}\binom{|X|}{i}\nonumber\\
   &-\sum_{i=0}^{r-k}(-1)^{r-i-k}\binom{r-i}{k}\sum_{j=0}^{i}\binom{|X|-j}{i-j}|\mathcal{C}_{j}|\nonumber\\
	&+y\sum_{i=0}^{r+1-k}(-1)^{r-i+1-k}\binom{r-i+1}{k}\sum_{j=0}^{i}\binom{|X|-j}{i-j}|\mathcal{C}_{j}|\nonumber\\
   =&\sum_{i=0}^{r-k}(-1)^{r-i-k}\binom{r-i}{k}\binom{|X|}{i}\nonumber\\
   &-\sum_{j=0}^{r-k}\sum_{i=j}^{r-k}(-1)^{r-i-k}\binom{r-i}{k}\binom{|X|-j}{i-j}|\mathcal{C}_{j}|\nonumber\\
	&+y\sum_{j=0}^{r-k+1}\sum_{i=j}^{r-k+1}(-1)^{r-i+1-k}\binom{r-i+1}{k}\binom{|X|-j}{i-j}|\mathcal{C}_{j}|\nonumber \\
   =& \binom{|X|-k-1}{r-k}-\sum_{j=0}^{r-k}\binom{|X|-k-1-j}{r-k-j}|\mathcal{C}_{j}|\nonumber\\
   &+y\sum_{j=0}^{r-k+1}\binom{|X|-k-1-j}{r-k-j+1}|\mathcal{C}_{j}|.\nonumber
\end{align}

\noindent (2). We start from Equation (\ref{Indset1}), we only consider independent sets $A$ with $i$ elements such that $r-i\geq k$.
If $k\geq r-g_2(M)+1$, then $r-i\geq k\geq r-g_2(M)+1$, that is, $ i\leq g_2(M)-1$. For any $A\subseteq X$, if $|A|\leq g_2(M)-1$, then $|A|-rk(A)\leq 1$. Otherwise $k=|A|-rk(A)\geq 2$, then $g_2(M)\leq g_k(M)\leq |A|$, a contradiction.

\noindent \textbf{Claim 3.}
\[\tau_{i}(M)=\binom{|X|}{i}-\sum_{C\in \mathcal{C}(M)}\binom{|X|-|C|}{i-|C|}=\binom{|X|}{i}-\sum_{j=0}^{i}\binom{|X|-j}{i-j}|\mathcal{C}_{j}|.\]
\begin{proof}
For any independent set $A$,
	suppose that there exists $A'\subseteq X$ with $|A'|=|A|$ such that $|A'|-rk(A')\geq 2$. Then
	\[rk(A')\leq |A'|-2=|A|-2=rk(A)-2. \]
	Therefore $r-rk(A')\geq r-rk(A)+2\geq k+2$. Since $k\geq r-g_2(M)+1$, we have $rk(A')\leq g_2(M)-3$. By Lemma \ref{boundrank}, we have $|A'|-rk(A')\leq 1$, a contradiction. Therefore, for any $A'\subseteq X$ with $|A'|=|A|$, we have $|A'|-rk(A')\leq 1$. Therefore, the number of $A'$'s with $|A'|=|A|=i$ such that $A'$ is an independent set is equal to  $\binom{|X|}{i}-\sum_{C\in \mathcal{C}(M)}\binom{|X|-|C|}{i-|C|}$.
	Thus the claim holds.
\end{proof}
By Equation (\ref{Indset1}), \textbf{Claim 3}, exchanging summands and Lemma \ref{identical1}, we have
\begin{align}\label{CX0}	
[x^k]T_M(x,1)=&\sum_{i=0}^{r-k}(-1)^{r-i-k}\binom{r-i}{k}\tau_{i}(M) \nonumber\\
	=&\sum_{i=0}^{r-k}(-1)^{r-i-k}\binom{r-i}{k}\binom{|X|}{i}\nonumber\\
    -&\sum_{i=0}^{r-k}(-1)^{r-i-k}\binom{r-i}{k}\sum_{j=0}^{i}\binom{|X|-j}{i-j}|\mathcal{C}_{j}| \nonumber\\
   =&\sum_{i=0}^{r-k}(-1)^{r-i-k}\binom{r-i}{k}\binom{|X|}{i}\nonumber\\
   -&\sum_{j=0}^{r-k}\sum_{i=j}^{r-k}(-1)^{r-i-k}\binom{r-i}{k}\binom{|X|-j}{i-j}|\mathcal{C}_{j}|\nonumber\\
   =&\binom{|X|-k-1}{r-k}-\sum_{j=0}^{r-k}\binom{|X|-k-1-j}{r-k-j}|\mathcal{C}_{j}|.   \nonumber
\end{align}

\noindent (3) and (4). Note that $g_k(M)+f_k(M^{*})=|X|$ and $C\in \mathcal{C}_j(M^{*})$ if and only if $X\setminus C\in \mathcal{H}_{|X|-j}(M)$. By Equation (\ref{dual}), (3) and (4) follow from (1) and (2) respectively.
\end{proof}

\section{Proof of Theorem \ref{d3}}

Let $M=(X, rk)$ be a matroid. A maximal independent set in a matroid $M$ is called a \textit{base} of $M$. Given a total order on $X$ and let $B$ be a base.  For $e \in X\setminus B$, $e$ is called to be \textit{externally active} with respect to $B$ if there exists no an element $f\in B$ which is smaller than $e$ such that $B-f+e$ is a base. For $e \in B$, $e$ is called to be  \textit{internally active} with respect to $B$ if there exists no an element $f\in X\setminus B$ which is smaller than $e$ such that $B-e+f$ is a base. The number of externally  (resp. internally) active elements with respect to $B$ is called \textit{external}
(resp. \textit{internal}) \textit{activity} of $B$, denoted by $\epsilon(B)$ (resp. $\iota(B)$).
It is well known that the Tutte polynomial of a matroid $M$ can be written as
\begin{eqnarray}
T_{M}(x,y)=\sum_{B\in \mathcal{B}(M)}x^{\iota(B)}y^{\epsilon(B)},
\end{eqnarray}
where $\mathcal{B}(M)$ is the set of all bases of $M$, which was originally introduced for graphs by Tutte \cite{Tutte} and later extended to matroids by Crapo \cite{Crapo}.  
An element $e$ is called to be a \textit{loop} (resp. \textit{coloop}) if $rk(e)=0$ (resp. $rk(M\setminus e)=rk(M)-1$). Then the Tutte polynomial of a matroid $M$ satisfies the following recursive relation.
$T_M(x,y)=1$ if $X=\emptyset$, otherwise, for $e\in X$,
\begin{eqnarray}
T_M(x, y)=\left \{
\begin{array}{ll}
xT_{M/e}(x, y),  &\text{ if } e \text{ is a coloop};\\
yT_{M\setminus e}(x, y),  &\text{ if } e \text{ is a loop};\\
T_{M\setminus e}(x, y)+T_{M/e}(x, y), &\text{ otherwise,}
\end{array}
\right.
\end{eqnarray}
where $M\setminus e$ and $M/e$ denote the deletion and contraction of $e$ from $M$, respectively. We need the following two lemmas to prove Theorem \ref{d3}.

\begin{lemma}\label{MO}
Let $M_1$ and $M_2$ be two matroids on same ground set $X$ such that $\mathcal{B}(M_1)\subseteq\mathcal{B}(M_2)$. Then, for any integer $k$ and any real number $t$ with $0\leq t\leq 1$,
\begin{eqnarray}
&[x^k]T_{M_1}(x,t)\leq [x^k]T_{M_2}(x,t),\label{fi}\\
&[y^k]T_{M_1}(t,y)\leq [y^k]T_{M_2}(t,y).\label{se}
\end{eqnarray}
\end{lemma}

\begin{proof}
We only prove (\ref{fi}),since (\ref{se}) can be obtained from (\ref{fi}) by the fact that $B\in \mathcal{B}(M)$ if and only if $X-B\in \mathcal{B}(M^*)$ using duality (Equation (\ref{dual})). For convenience, let $l(M)$ and $c(M)$ denote the number of loops and coloops for a matroid $M$. There are two cases.

\noindent{\bf Case 1.} $l(M_1)+c(M_1)=|X|$. In this case, $M_1$ contains a unique base $B$ which consists of all coloops. For any fixed order on $X$, each loop is externally active and each coloop is internally active. Therefore, $T_{M_1}(x,t)=t^{l(M_1)}x^{c(M_1)}$. Since $\mathcal{B}(M_1)\subseteq\mathcal{B}(M_2)$, we have $B\in \mathcal{B}(M_2)$ and $l(M_2)\leq l(M_1)$. Now we order elements of $B$ first then elements of $X-B$. Thus every element in $B$ is internally active with respect to $B$ in $M_2$, every loop in $M_2$ is remains externally active with respect to $B$, and non-loop elements in $X-B$ are not external active with respect to $B$ in $M_2$. It follows that $[x^{c(M_1)}]T_{M_2}(x,t)\geq t^{l(M_2)}$. Since $0\leq t\leq 1$, we have $[x^k]T_{M_1}(x,t)\leq [x^k]T_{M_2}(x,t)$.

\noindent{\bf Case 2.} $l(M_1)+c(M_1)<|X|$. We prove this case by induction on $|X|$. If $|X|= 2$, then both $M_1$ and $M_2$ are circuits with two elements and hence (\ref{fi}) holds. Assume that (\ref{fi}) holds for any matroid on the ground set of size strictly less than $|X|$ with $|X|\ge 3$. Now suppose $|X|$. Since $l(M_1)+c(M_1)<|X|$, there exists an element $e\in X$
which is neither a loop nor a coloop in $M_1$.
Since $\mathcal{B}(M_1)\subseteq\mathcal{B}(M_2)$, each of the following holds:
\begin{itemize}
  \item [(i)] $e$ is neither a loop nor a coloop in $M_2$;
  \item [(ii)] $\mathcal{B}(M_1\setminus e)\subseteq\mathcal{B}(M_2\setminus e)$;
   \item [(iii)] $\mathcal{B}(M_1/ e)\subseteq\mathcal{B}(M_2/ e)$.
\end{itemize}
Thus, by induction hypothesis and Case 1, we have
\begin{align*}	
    [x^k]T_{M_1}(x,t)& =[x^k]T_{M_1\setminus e}(x,t)+[x^k]T_{M_1/ e}(x,t)\\
    &\leq [x^k]T_{M_2\setminus e}(x,t)+[x^k]T_{M_2/ e}(x,t)\\
    &=[x^k]T_{M_2}(x,t).
\end{align*}

This completes the proof of Lemma \ref{MO}.  
\end{proof}

Recall that the matroid on a finite set of $n$ elements in which every subset of size at most $r$ is independent is a \textit{uniform matroid}, denoted by $U_{r,n}$, where $0\leq r\leq n$.
The Tutte polynomial of uniform matroids was given in \cite{MerinoRa} (see Formula (2.24)).

\begin{lemma}\cite{MerinoRa}\label{TutteU}
\[
T_{U_{r,n}}(x, y) = \sum_{i=1}^r \binom{n-i-1}{n-r-1}x^i + \sum_{j=1}^{n-r}\binom{n-j-1}{r-1}y^j
\]
if \( 0 < r < n \), and \( T_{U_{n,n}}(x, y) = x^n \) and \( T_{U_{0,n}}(x, y) = y^n \).
\end{lemma}

We now present the proof of Theorem \ref{d3}.

\begin{proof} We only prove (1) and (2), since (3) and (4) can be obtained from (1) and (2) by duality.
If $r=0$, then
$$T_{M}(x,y)=y^{|X|}.$$
If $r=|X|-1$, then
$$T_{M}(x,y)=x^{|X|-1}+x^{|X|-2}+\cdots +x^{|X|-|C|+1}+x^{|X|-|C|}y,$$
where $C$ denotes the unique circuit in $M$.
It is easy to check that (1) and (2) hold for $r=0$ and $r=|X|-1$. We now assume that $1\leq r\leq |X|-2$.

We first establish the necessity of (1) and (2) by Theorem \ref{d2}. Note that $k\geq r-g(M)+2\geq r-g_2(M)+3$. If $j\leq r-k+1$, then $j\leq g(M)-1$ Which implies that $\mathcal{C}_{j}=\emptyset.$ Thus, by (1) of Theorem \ref{d2}, we have
	\begin{align*}	
[x^k]T_M(x,t)=\binom{\alpha}{r-k}.
	\end{align*}
Similarly, if $k\geq r-g(M)+1$ and $j\leq r-k$, then $ j\leq g(M)-1$ which implies $\mathcal{C}_{j}=\emptyset.$ Thus,  by (2) of Theorem \ref{d2}, we have
	\begin{align*}	
[x^k]T_M(x,1)=\binom{\alpha}{r-k}.
	\end{align*}

We now prove the sufficiency of (1) and (2).

\noindent {\bf Case 1.}  $M$ contains a loop $e$. Clearly, $g(M)=1$ and  $T_M(x,t)=tT_{M\setminus e}(x,t)$. Since $r\leq (|X|-1)-1$,
we, by Lemmas \ref{MO} and \ref{TutteU}, have
$$[x^k]T_{M\setminus e}(x,t)\leq [x^k]T_{U_{r,|X|-1}}(x,t)=\binom{|X|-1-k-1}{r-k}\leq \binom{|X|-k-1}{r-k}.$$  Therefore $$[x^k]T_{M}(x,t)=t[x^k]T_{M\setminus e}(x,t)<\binom{|X|-k-1}{r-k}$$ if $0\leq t< 1$ and $r-k\geq 0$, and $$[x^k]T_{M}(x,1)=[x^k]T_{M\setminus e}(x,1)\leq \binom{|X|-1-k-1}{r-k}<\binom{|X|-k-1}{r-k}$$ if $r-k\geq 1$. This contradicts $[x^k]T_{M}(x,t)=\binom{|X|-k-1}{r-k}$. Hence, $r-k\leq -1$ if $0\leq t< 1$, and $r-k\leq 0$ if $t=1$, that is, $k\geq r+1=r-g(M)+2$ if $0\leq t< 1$, and $k\geq r-g(M)+1$ if $t=1$.

\noindent {\bf Case 2.} $M$ contains no loops. We prove this case by induction on $|X|$. If $|X|=3$, then we have $r=1$, $g(M)=2$ and $T_M(x,t)=t^2+t+x$. Clearly, the sufficiency of (1) and (2) hold.
We assume that the sufficiency of (1) and (2) for any matroid on the ground set of size strictly less than $l\geq 4$. Now suppose that $|X|=l$.
Since $|X|\geq r+2$ and $M$ contains no loops, $M$ must contain a minimum circuit with at least two elements. Let $e$ be an element in this minimum circuit. Clearly, $e$ is neither a loop nor a coloop. By Lemmas \ref{MO} and \ref{TutteU}, we have $$[x^k]T_{M/ e}(x,t)\leq [x^k]T_{U_{r-1,|X|-1}}(x,t)=\binom{|X|-1-k-1}{r-1-k}$$ and
$$[x^k]T_{M\setminus e}(x,t)\leq [x^k]T_{U_{r,|X|-1}}(x,t)=\binom{|X|-1-k-1}{r-k}.$$
Then
\begin{align*}
[x^k]T_M(x,t)&=[x^k]T_{M/ e}(x,t)+[x^k]T_{M\setminus e}(x,t)\\
&\leq \binom{|X|-1-k-1}{r-1-k}+\binom{|X|-1-k-1}{r-k}\\
&=\binom{|X|-k-1}{r-k}\\
&=[x^k]T_M(x,t).
\end{align*}
Therefore, $$[x^k]T_{M/ e}(x,t)=\binom{|X|-1-k-1}{r-1-k}. $$
If $M/e$ contains no loops, then, by the induction hypothesis, we have
\begin{eqnarray}\label{kin}
k\geq\left \{
\begin{array}{ll}
rk(M/e)-g(M/e)+2,  &\text{ if } 0\leq t<1;\\
rk(M/e)-g(M/e)+1, &\text{ if } t=1.
\end{array}
\right.
\end{eqnarray}
Otherwise, Inequation (\ref{kin}) can be obtained by replacing $M$ with $M/ e$ in {\bf Case 1.} Note that $rk(M/e)=rk(M)-1$ and $g(M/e)=g(M)-1$.
Thus,
\begin{eqnarray*}
k\geq\left \{
\begin{array}{ll}
rk(M)-g(M)+2,  &\text{ if } 0\leq t<1;\\
rk(M)-g(M)+1, &\text{ if } t=1.
\end{array}
\right.
\end{eqnarray*}
\end{proof}

\section{Applications}

In this section, we discuss applications of Theorems \ref{d2} and \ref{d3}.

\subsection{$t=0$}

The \textit{characteristic polynomial} $\chi_M(\lambda)$ of a matroid $M$, a special case of Tutte polynomial, can be defined by
\begin{equation}\label{cha}
\chi_M(\lambda)=(-1)^{r}T_M(1-\lambda,0).
\end{equation}

Let $G=(V,E)$ denote a graph of order $n$ and size $m$, i.e. $|V|=n$ and $|E|=m$. For any $A\subseteq E$, let $c(A)$ denote the number of components of the spanning subgraph $(V,A)$. In particular, we write $c=c(E)$ for simplicity.
For a graph $G=(V,E)$, the matroid on the ground set $E$ with the rank function $rk(A)=|V|-c(A)$ for any $A\subseteq E$ is called the \textit{cycle matroid} of $G$, denoted by $M(G)$. The Tutte polynomial $T_G(x,y)$ of a graph $G$ can be defined by $T_{M(G)}(x,y)$, that is,
\begin{equation*}
T_G(x,y)=T_{M(G)}(x,y).
\end{equation*}
The classical \textit{chromatic polynomial} $\chi_G(\lambda)$ and \textit{flow polynomial} $F_G(\lambda)$ of a graph $G$ can be expressed by the Tutte polynomial of $G$ as follows:
\begin{align*}
\chi_G(\lambda)&=(-1)^{n-c}\lambda^{c}T_G(1-\lambda,0);\\
F_G(\lambda)&=(-1)^{m-n+c}T_G(0,1-\lambda).
\end{align*}
Clearly, $\chi_G(\lambda)=\lambda^{c}\ \chi_{M(G)}(\lambda).$

For a connected simple graph $G$ of order $n$, Whitney \cite{Whitney}  proved the classical result that $(-1)^k[\lambda^{n-k}]\chi_G(\lambda)$
counts the number of $k$-subset of edges containing no broken cycles for $0\leq k<n$.
In 1994, Teo and Koh \cite{Teo} further proved:
\begin{theorem}\emph{\cite{Teo}}\label{TK}
Let $G$ be a graph of order $n$ and size $m$ with girth $g\geq 4$.
Then, for each integer $k$ with $g\leq k\leq \lceil \frac{3g}{2} \rceil -3$,
\[(-1)^k[\lambda^{n-k}]\chi_G(\lambda)=\binom{m}{k}-\sum_{i=g}^{k+1}\phi_{i}(G)\binom{m-i+1}{k-i+1},\]
where $\phi_{i}(G)$ denotes the number
of cycles of length $i$ in $G$.
\end{theorem}

As an application of Theorem \ref{d2}, we will obtain the matroid version of Theorem \ref{TK}, as presented in Theorem \ref{chra} below. In the process, the following identity is required, which is not new and can be derived from Equation (3b) on Page 8 in \cite{Riordan}.

\begin{lemma}\label{identical2}
	Let $m$ be a nonnegative integer. Then, for any nonnegative integers $p$ and $k$ with $k\leq q$,
\[\sum_{i=0}^{q}\binom{m-i-1}{q-i}\binom{i}{k}=\binom{m}{q-k}.\]
\end{lemma}

\begin{theorem}\label{chra}
Let $M=(X,rk)$ be a matroid.  Then, for
	$k\geq r-g_2(M)+3$,
	\begin{align*}		
(-1)^{r+k}[x^k]\chi_M(\lambda)=\binom{|X|}{r-k}-\sum_{j=0}^{r-k+1}\binom{|X|-j+1}{r-k-j+1}|\mathcal{C}_{j}|.
	\end{align*}
\end{theorem}

\begin{proof}
Let $T_M(x,0)=\sum_{i=0}^{r} a_{i}x^{i}$. Then
\begin{align*}	
T_M(1-x,0)=\sum_{i=0}^{r} a_{i}(1-x)^{i}=\sum_{i=0}^{r} a_{i}\sum_{j=0}^{i}\binom{i}{j}(-x)^{j}=\sum_{j=0}^{r}(-1)^{j}x^{j}\sum_{i=j}^{r}a_{i}\binom{i}{j}.
\end{align*}
Clearly, $$[x^k]T_M(1-x,0)=(-1)^{k}\sum_{i=k}^{r}a_{i}\binom{i}{k}.$$
For $i\geq k\geq r-g_2(M)+3$, by Theorem \ref{d2} (1), we have
\[
\begin{aligned}
a_i &= \binom{|X|-i-1}{r-i} - \sum_{j=0}^{r-i} \binom{|X|-i-1-j}{r-i-j} |\mathcal{C}_j| - \sum_{j=0}^{r-i+1} \binom{|X|-i-1-j}{r-i-j+1} |\mathcal{C}_j| \\
&= \binom{|X|-i-1}{r-i} - \sum_{j=0}^{r-i} \binom{|X|-i-1-j}{r-i-j} |\mathcal{C}_j| - \sum_{j=0}^{r-i} \binom{|X|-i-1-j}{r-i-j+1} |\mathcal{C}_j| \\
&\quad - \binom{|X|-i-1-(r-i+1)}{r-i+1-(r-i+1)} |\mathcal{C}_{r-i+1}| \\
&= \binom{|X|-i-1}{r-i} - \sum_{j=0}^{r-i} \binom{|X|-i-1-j+1}{r-i-j+1} |\mathcal{C}_j| - |\mathcal{C}_{r-i+1}| \\
&= \binom{|X|-i-1}{r-i} - \sum_{j=0}^{r-i} \binom{|X|-i-j}{r-i-j+1} |\mathcal{C}_j| \\
&\quad - \binom{|X|-i-(r-i+1)}{r-i+1-(r-i+1)} |\mathcal{C}_{r-i+1}| \\
&= \binom{|X|-i-1}{r-i} - \sum_{j=0}^{r-i+1} \binom{|X|-i-j}{r-i-j+1} |\mathcal{C}_j|.
\end{aligned}
\]
Therefore $[x^k]T_M(1-x,0)$
\begin{align*}	
    =&(-1)^{k}\sum_{i=k}^{r}\left(\binom{|X|-i-1}{r-i}-\sum_{j=0}^{r-i+1}\binom{|X|-i-j}{r-i-j+1}|\mathcal{C}_{j}|\right)\binom{i}{k}\\	=&(-1)^{k}\left(\sum_{i=k}^{r}\binom{|X|-i-1}{r-i}\binom{i}{k}-\sum_{i=k}^{r}\sum_{j=0}^{r-i+1}\binom{|X|-i-j}{r-i-j+1}\binom{i}{k}|\mathcal{C}_{j}|\right)\\
=&(-1)^{k}\left(\sum_{i=k}^{r}\binom{|X|-i-1}{r-i}\binom{i}{k}-\sum_{j=0}^{r-k+1}\sum_{i=k}^{r-j+1}\binom{|X|-j-i}{r-j-i+1}\binom{i}{k}|\mathcal{C}_{j}|\right)\\
=&(-1)^{k}\left(\binom{|X|}{r-k}-\sum_{j=0}^{r-k+1}\binom{|X|-j+1}{r-k-j+1}|\mathcal{C}_{j}|\right),
\end{align*}
where the last equality follows from Lemma \ref{identical2}. By (\ref{cha}), Theorem \ref{chra} holds.
\end{proof}

To show that Theorem \ref{chra} generalizes Theorem \ref{TK}, we first prove the following lemma.

\begin{lemma}\label{girth}
	Let $G$ be a graph satisfying $m-n+c\geq 2$ with girth $g$. Then
	$$g_2(M(G))\geq \frac{3g}{2}.$$	
\end{lemma}
\begin{proof}
Let $E'$ be an edge subset realizing $g_2(M(G))$. Then $|E'|-rk(E')=2$ implying the graph induced by $E'$ has two distinct cycles $C_{1}$ and $C_{2}$. Assume that $g_2(M(G))<\frac{3g}{2}$. Then $|E(C_{1})\setminus E(C_{2})|\leq|E'\setminus E(C_{2})|<\frac{3g}{2}-g=\frac{g}{2}$. Similarly $|E(C_{2})\setminus E(C_{1})|< \frac{g}{2}$. Let $E''=E(C_{1})\vartriangle E(C_{2})$, the symmetric difference of the two cycles. Then $|E''|<g$ contradicting to the fact that the graph induced by $E'$ has a cycle.
\end{proof}

Suppose $k\leq \lceil \frac{3g}{2} \rceil -3$.
By Lemma \ref{girth}, we have $n-k\geq n-\lceil \frac{3g}{2} \rceil +3\geq n-g_2(M(G)) +3$. Then $$n-c-k\geq n-c-g_2(M(G))+3=rk(M(G))-g_2(M(G))+3.$$
By Theorem \ref{chra}, we have
\begin{align*}		
[\lambda^{n-k}]\chi_G(\lambda)=&[\lambda^{n-c-k}]\chi_{M(G)}(\lambda)\\
=&(-1)^{k}\left(\binom{m}{k}-\sum_{j=0}^{k+1}\binom{m-j+1}{k-j+1}|\mathcal{C}_{j}|\right).
	\end{align*}

Likewise, applying Theorem \ref{d3}, we deduce the following result.

\begin{theorem}\label{Chiff}
Let $M=(X,rk)$ be a matroid with $r< |X|$.
 Then $g(M)\geq r-k+2$ if and only if
	\begin{align*}		
[x^k]\chi_M(\lambda)=(-1)^{r+k}\binom{|X|}{r-k}.
	\end{align*}
\end{theorem}

\begin{proof}
Let $T_M(x,0)=\sum_{i=1}^{r} a_{i}x^{i}$. By Lemmas \ref{MO} and \ref{TutteU}, we have $a_i\leq \binom{|X|-i-1}{r-i}$ for any $1\leq i\leq r$.
Since
\begin{align*}	
T_M(1-x,0)=\sum_{i=1}^{r} a_{i}(1-x)^{i}=\sum_{i=1}^{r} a_{i}\sum_{j=0}^{i}\binom{i}{j}(-x)^{j}=\sum_{j=0}^{r}(-1)^{j}x^{j}\sum_{i=j}^{r}a_{i}\binom{i}{j},
\end{align*}
we have
\begin{align*}	
    (-1)^{r+k}[x^k]\chi_M(\lambda)=&(-1)^{k}[x^{k}]T_M(1-x,0) \\
	=&\sum_{i=k}^{r}a_{i}\binom{i}{k}\\	
    \leq &\sum_{i=k}^{r}\binom{|X|-i-1}{r-i}\binom{i}{k}\\	
=&\binom{|X|}{r-k},
\end{align*}
where the last equation follows from Lemma \ref{identical2}. Therefore $(-1)^{r+k}[x^k]\chi_M(\lambda)=\binom{|X|}{r-k}$ if and only if $a_i= \binom{|X|-i-1}{r-i}$. Thus Theorem \ref{Chiff} holds from Theorem \ref{d3}.
\end{proof}

The following corollary is immediate. Its ``only if'' part appeared in classical textbook on algebraic graph theory, say, Proposition 10.6 in \cite{Biggs}.

\begin{corollary}\label{graphChrff}
Let $G$ be a graph of order $n$ and size $m$ containing a cycle.
 Then $g(G)\geq n-k+2$ if and only if
	\begin{align*}		
[\lambda^k]\chi_G(\lambda)=(-1)^{n+k}\binom{m}{n-k}.
	\end{align*}
\end{corollary}

Dually, we have:

\begin{corollary}\label{graphflowC}
Let $G$ be a connected graph of order $n\geq 2$, size $m$ and edge connectivity $\kappa')$.
 Then $\kappa'\geq m-n-k+3$ if and only if
	\begin{align*}		
[\lambda^k]F_G(\lambda)=(-1)^{m-n+k+1}\binom{m}{n+k-1}.
	\end{align*}
\end{corollary}

\begin{proof}. By Theorem \ref{Chiff} and the relation $F_G(\lambda)=\chi_{(M(G))^{*}}(\lambda)$, we have
$f_1(M(G))\leq n+k-3$ if and only if
	\begin{align*}		
[\lambda^k]F_G(\lambda)=(-1)^{m-n+k+1}\binom{m}{n+k-1}.
	\end{align*}
Since $m-\kappa'(G)=f_1(M(G))$, $\kappa'(G)\geq m-n+3-k$
if and only if $f_1(M(G))\leq n+k-3$. Thus the corollary holds.
\end{proof}

\subsection{$t=1$}

In 2013, the polynomials $T_G(x,1)$ and $T_G(1,y)$ were generalized to hypergraphs by K\'{a}lm\'{a}n \cite{Kalman}, called
interior and exterior polynomials, and further extended to integer polymatroids. Based on the definition of Tutte polynomials in terms of internal and external activities, Bernardi et al. \cite{Bernardi} defined a new Tutte polynomial for integer polymatroids, which combines interior and exterior polynomials into a bivariate polynomial. In \cite{Guan0}, Guan et al. studied coefficients of polymatroid Tutte polynomial with one variable fixed, including the interior and exterior polynomials. Guan et al. \cite{Guan}
established the following partial coefficients of Tutte polynomials with the variable $y$ fixed at $1$ for $(k+1)$-edge connected graphs, which were deduced from coefficients of the exterior polynomial of polymatroids. 

\begin{theorem}\emph{\cite{Guan}}\label{GJK}
Let $G$ be a graph of order $n$ and size $m$.
Then, if $G$ is $(k+1)$-edge connected if and only if for  $m-n-k+1\leq j\leq m-n+1$,
\[[y^j]T_G(1,y)=\binom{m-j-1}{n-2}.\]
\end{theorem}

Recently, Chen and Guo \cite{Chen} generalized Theorem \ref{GJK} by applying a result of Merino \cite{Merino} that is closely related to the Abelian sandpile model.

\begin{theorem}\emph{\cite{Chen}}\label{CG}
Let $G$ be a $(k+1)$-edge connected graph of order $n$ and size $m$.
Then, for $m-n+1-\frac{3(k+1)}{2}< j\leq m-n+1$,
\[[y^j]T_G(1,y)=\binom{m-j-1}{n-2}-\sum_{i=k+1}^{m-n-j+1}\binom{m-j-i-1}{n-2}|\mathcal{EC}_i(G)|,\]
where
$\mathcal{EC}_i(G)$ denotes the set of all minimal edge cuts with $i$ edges.	
\end{theorem}

In the following we show that Theorem \ref{CG} can be derived from Theorem \ref{d2} (4) using the following lemma.

\begin{lemma}\label{cut}
	Let $G$ be a $k$-edge-connected graph with $m$ edges. Then
	$$f_2(M(G))\leq m-\frac{3k}{2}.$$	
\end{lemma}
\begin{proof}
	Let $S$	be a minimal set of edges of $G$ such that $G-S$ has three connected components, say $G_1$, $G_2$ and $G_3$. Let $S_{ij}$ be the set of edges connecting $G_i$ and $G_j$. Since $G$ is $k$-edge-connected, it follows that $|S_{12}|+|S_{13}|\geq k$, $|S_{12}|+|S_{23}|\geq k$ and $|S_{13}|+|S_{23}|\geq k$. Note that $|S|=|S_{12}|+|S_{23}|+|S_{13}|$. Therefore $|S|\geq \frac{3k}{2}$. Since $E(G)\setminus S$ is a flat with rank $n-3=rk(M(G))-2$ in $M(G)$,  we have
	\[
	f_2(M(G))\leq m-\frac{3k}{2}.\qedhere
	\]
\end{proof}

If $G$ is $(k+1)$-edge-connected, then,
by Lemma \ref{cut}, we have
\[f_2(M(G))-rk(M(G))=f_2(M(G))-n+1\leq m-\frac{3(k+1)}{2}-n+1.\]
Note that minimal edge cuts of $G$ correspond to hyperplanes of $M(G)$. So Theorem \ref{d2} (4) implies Theorem \ref{CG}.

For a graph $G$, $g_2(M(G))$ is the size of the minimum subgraph of $G$ containing two cycles. For brevity, we write $g_2(G)$ for $g_2(M(G))$. By Theorem \ref{d2} (2), we have the following dual result of Theorem \ref{CG}:

\begin{corollary}\label{dcg}
	Let $G$ be a graph of order $n$ and size $m$ containing two cycles. Then, for
	$i\geq n-c(G)-g_2(G)+1$,
	\[[x^i]T_G(x,1)=\binom{m-i-1}{n-c(G)-i}-\sum_{j=1}^{n-c(G)-i} \binom{m-i-1-j}{m-n+c(G)-1}|\mathcal{C}_j(G)|,\]
where
$\mathcal{C}_j(G)$ denotes the set of cycles with $j$ edges.
\end{corollary}

\subsection{Unimodality}

A  sequence $a_{0}$,$a_{1}$,$\ldots$,$a_{n}$ of real numbers is called to be \emph{unimodal} if there exists an integer $k$ such that $a_{i}\leq a_{i+1}$ for any $i<k$ and $a_{i}\geq a_{i+1}$ for any $i>k$.
As another application of Theorem \ref{d2}, we have:
\begin{theorem}\label{Uni}
	Let $M=(X,rk)$ be a matroid.  Then, for any $t$ with $t\leq 1$, the following two sequences are unimodal:
\begin{itemize}
  \item [(1)] $[x^{r}]T_M(x,t), [x^{r-1}]T_M(x,t),\ldots,[x^{r-g_{2}(M)+3}]T_M(x,t);$
  \item [(2)] $[y^{|X|-r}]T_M(t,y),[y^{|X|-r-1}]T_M(t,y),\ldots, [y^{f_{2}(M)-r+3}]T_M(t,y)$.
\end{itemize}
\end{theorem}

\begin{proof}

By duality we only need to prove (1). If $[x^{i-1}]T_M(x,t'+1)-[x^i]T_M(x,t'+1)<0$ for any $i\in [r-g_{2}(M)+4, r]$, then the conclusion is clear. Therefore,
it is sufficient to prove that if there exists an integer $i\in [r-g_{2}(M)+4, r]$ such that $[x^{i-1}]T_M(x,t'+1)-[x^i]T_M(x,t'+1)\geq 0$, then $[x^{i}]T_M(x,t'+1)-[x^{i+1}]T_M(x,t'+1)\geq 0$.

If $r=|X|$, then $T_M(x,y)=x^{r}$. The conclusion holds.
If $r=|X|-1$, then $T_M(x,y)=(x+x^{2}+\cdots+x^{|C|-1}+y)y^{|X|-|C|}$, where $C$ is the unique circuit of $M$. The conclusion holds.
We now assume that $r\leq m-2$. By Theorem \ref{d2} (1), for any $k\in [r-g_2(M)+4,r]$, we have
\begin{align*}	
    & [x^{k-1}]T_M(x,t'+1)-[x^k]T_M(x,t'+1) \\
=&\binom{\alpha}{r-k+1}-\sum_{j=0}^{r-k+1}\binom{\alpha-j}{r-k+1-j}|\mathcal{C}_{j}|
   +t'\sum_{j=0}^{r-k+2}\binom{\alpha-j}{r-k-j+2}|\mathcal{C}_{j}|.
\end{align*}
Using the identity $\binom{b-1}{a-1}=\frac{a}{b}\binom{b}{a}$ and the inequality $\frac{a-j}{b-j}\leq \frac{a}{b}$ for any $b\geq a\geq j\geq 0$, we have
\begin{align*}	
    & [x^{k}]T_M(x,t'+1)-[x^{k+1}]T_M(x,t'+1) \\
    =&\binom{\alpha-1}{r-k}-\sum_{j=0}^{r-k}\binom{\alpha-j-1}{r-k-j}|\mathcal{C}_{j}|
   +t'\sum_{j=0}^{r-k+1}\binom{\alpha-j-1}{r-k-j+1}|\mathcal{C}_{j}| \\
=&\binom{\alpha-1}{r-k}-\sum_{j=0}^{r-k+1}\binom{\alpha-j-1}{r-k-j}|\mathcal{C}_{j}|
   +t'\sum_{j=0}^{r-k+2}\binom{\alpha-j-1}{r-k-j+1}|\mathcal{C}_{j}|\\	=&\frac{r-k+1}{\alpha}\binom{\alpha}{r-k+1}-\sum_{j=0}^{r-k+1}\frac{r-k+1-j}{\alpha-j}\binom{\alpha-j}{r-k+1-j}|\mathcal{C}_{j}|\\	&+t'\sum_{j=0}^{r-k+2}\frac{r-k-j+2}{\alpha-j}\binom{\alpha-j}{r-k-j+2}|\mathcal{C}_{j}|\\
   \geq&\frac{r-k+1}{\alpha}\binom{\alpha}{r-k+1}-\sum_{j=0}^{r-k+1}\frac{r-k+1}{\alpha}\binom{\alpha-j}{r-k+1-j}|\mathcal{C}_{j}|\\	&+t'\sum_{j=0}^{r-k+2}\frac{r-k+1}{\alpha}\binom{\alpha-j}{r-k-j+2}|\mathcal{C}_{j}|\\
   \geq&\frac{r-k+1}{\alpha}\left([x^{k-1}]T_M(x,t'+1)-[x^k]T_M(x,t'+1)\right).
\end{align*}
Thus, if $[x^{k-1}]T_M(x,t'+1)-[x^k]T_M(x,t'+1)\geq 0$, then $[x^{k}]T_M(x,t'+1)-[x^{k+1}]T_M(x,t'+1)\geq 0$.
This completes the proof of (1).
\end{proof}

\section*{Acknowledgements}

We would like to thank Dr. Meiqiao Zhang for providing an elegant and concise
proof of the combinatorial identity in Lemma \ref{identical1}. We also thank the anonymous referees of the early version of the paper leading to the present paper.

This work is supported by National Natural Science Foundation of China
(Nos. 12571379, 12401462, 12571366), the Natural Science Foundation of Shanxi Province (No. 202403021222034) and Scientific Research Start-Up Foundation of Jimei University (No. ZQ2024116).

\section*{Declarations}
\noindent
{ \bf Conflict of interest}
The authors declare that they have no conflict of interest.


\end{document}